\documentclass[a4paper,10pt]{article}

\usepackage{amsmath}
\usepackage{amssymb}
\usepackage{amsthm}
\usepackage{mathrsfs}
\usepackage{color}
\usepackage[dvips]{graphicx}
\usepackage[all]{xy}

\theoremstyle{plain}
\newtheorem{them}{Theorem}[section]
\newtheorem{lemma}[them]{Lemma}

\newtheorem{coro}[them]{Corollary}
\theoremstyle{definition}

\newtheorem{conj}[them]{Conjecture}

\newtheorem*{mthem}{Main Theorem}

\begin{document}

\title{The zeta function of a finite category and the series Euler characteristic}
\author{Kazunori Noguchi \thanks{noguchi@math.shinshu-u.ac.jp}}
\date{}
\maketitle
\begin{abstract}
We prove that a certain conjecture holds true and the conjecture states a relationship between the zeta function of a finite category and the Euler characteristic of a finite category.
\end{abstract}


\thispagestyle{empty}

\section{Introduction}

In \cite{NogA}, the zeta function of a finite category was defined and one conjecture was proposed. The zeta function of a finite category $I$ is the formal power series defined by $$\zeta_I(z)=\exp\left( \sum_{m=1}^{\infty} \frac{\# N_m(I)}{m} z^m\right)$$
where $$N_m(I)=\{\xymatrix{(x_0\ar[r]^{f_1}&x_1\ar[r]^{f_2}&\dots\ar[r]^{f_m}&x_m)} \text{ in } I \}.$$ The conjecture states a relationship between the zeta function of a finite category and the Euler characteristic of a finite category, called \textit{series Euler characteristic} \cite{Leib}. 

\begin{conj}
Suppose $I$ is a finite category which has series Euler characteristic. Then, we have
\renewcommand{\theenumi}{C\arabic{enumi}}
\renewcommand{\labelenumi}{(\theenumi)}
\begin{enumerate}
\item the zeta function of $I$ is a finite product of the following form $$\zeta_I(z)=\prod \frac{1}{(1-\alpha_iz)^{\beta_i}}\exp\left( \sum \frac{\gamma_j z^j}{j(1-\delta_j z)^j}\right)$$
for some complex numbers $\alpha_i, \beta_i, \gamma_j,\delta_j $.
\item  $\displaystyle \sum \beta_i$ is the number of objects of $I$. 
\item each $\alpha_i$ is an eigen value of $A_I$. Hence, $\alpha_i$ is an algebraic integer.
\item $\displaystyle \sum \frac{\beta_i}{\alpha_i} +\sum (-1)^j\frac{\gamma_j}{\delta_j^{j+1}}=\chi_{\Sigma}(I).$
\end{enumerate}
\end{conj}
It was verified this conjecture holds true under certain additional conditions in \cite{NogA} and \cite{NogB}.

In \cite{NogA}, it was verified the conjecture holds true in concrete cases, that is, when a finite category is a groupoid, an acyclic category and has two objects and so on. An \textit{acyclic} category is a small category in which all endomorphisms and isomorphisms are identity morphisms. In \cite{NogB}, it was verified the conjecture holds true when a finite category has M\"obius inversion. A finite category $I$ has \textit{M\"obius inversion} if its adjacency matrix  $A_I$ has an inverse matrix where $A_I$ is an $N\times N$-matrix whose $(i,j)$-entry is the number of morphisms of $I$ from $x_i$ to $x_j$ when the set of objects of $I$ is $$\mathrm{Ob}(I)=\{x_1,x_2,\dots,x_N\}$$ (see \cite{Leia} and \cite{Leic}). In the sense of Leinster, this is called \textit{coarse M\"obius inversion} \cite{Leic}. The class of finite categories which has coarse M\"obius inversion is large and very important to consider the Euler characteristic of a finite category. Euler characteristic for categories is defined by various ways, \textit{the series Euler characteristic} $\chi_{\sum}$ \cite{Leib}, \textit{the $L^2$-Euler characteristic} $\chi^{(2)}$ \cite{FLS}, \textit{the extended $L^2$-Euler characteristic} $\chi^{(2)}_{\mathrm{ex}}$ \cite{Nog}, \textit{the Euler characteristic of an $\mathbb{N}$-filtered acyclic category} $\chi_{\mathrm{fil}}$ \cite{Nog11} and so on. If a finite category $I$ has the coarse M\"obius inversion, then $I$ has Leinster's Euler characteristic and series Euler characteristic and they coincide, $\chi_L(I)=\chi_{\sum}(I)$. A finite acyclic category $A$ has the coarse M\"obius inversion and all of the Euler characteristic above for $A$ coincide.

In this paper, we prove the conjecture holds true without any additional conditions. The following is our main theorem.

\begin{mthem}
Suppose $I$ has series Euler characteristic and $$\deg(|E-A_Iz|)=N-r$$  and $$\deg(\mathrm{sum}(\mathrm{adj}(E-A_Iz)))=N-1-s$$ and the polynomial $|E-A_Iz|$ is factored to the following form
$$|E-A_Iz|=d_{N-r}(z-\theta_1)^{e_1}\dots(z-\theta_n)^{e_n}$$ where each $e_i\ge 1$ and $\theta_i\not = \theta_j$ if $i\not = j$. Then the rational function $$\frac{\mathrm{sum}(\mathrm{adj}(E-A_I z)A_I)}{ |E-A_I z|}$$ has a partial fraction decomposition to the following form 
$$\frac{\mathrm{sum}(\mathrm{adj}(E-A_I z)A_I)}{ |E-A_I z|}=\frac{1}{d_{N-r} }\sum^n_{k=1}\sum^{e_k}_{j=1}\frac{A_{k,j}}{(z-\theta_k)^j}$$ 
for some complex numbers $A_{k,j}$.
Moreover, 
\begin{enumerate}
\item  Then the zeta function of $I$ is 
\begin{multline*}
\zeta_I(z)=\prod^n_{k=1}\frac{1}{(1-\frac{1}{\theta_k}z)^{-\frac{A_{k,1}}{ d_{N-r}}}}  \\ \times \exp\left( \frac{1}{ d_{N-r}}  \sum^n_{k=1} \sum^{e_k -1}_{j=1} \frac{z^j}{j(1-\frac{1}{\theta_k}z)^j}\bigg( \sum_{i=j}^{e_k-1} \binom{i-1}{j-1} (-1)^{i+1} \bigg(\frac{1}{\theta_k}\bigg)^{i+j} A_{k,i+1}\bigg)\right)
\end{multline*}
\item $\displaystyle \sum^n_{k=1}-\frac{A_{k,1}}{ d_{N-r}}=N$
\item  Each $\frac{1}{\theta_k}$ is an eigen value of $A_I$. In particular, $\frac{1}{\theta_k}$ is an algebraic integer
\item  
\begin{multline*}
\displaystyle \sum^n_{k=1}\frac{-\frac{A_{k,1}}{ d_{N-r}}}{\frac{1}{\theta_k}}\\
+ \frac{1}{ d_{N-r}} \sum^n_{k=1} \sum ^{e_k-1}_{j=1} (-1)^j\frac{ \sum_{i=j}^{e_k-1}  \binom{i-1}{j-1} (-1)^{i-1} \bigg(\frac{1}{\theta_k}\bigg)^{i+j} A_{k,i+1}}{\left(\frac{1}{\theta_k}\right)^{j+1}}=\chi_{\Sigma}(I).
\end{multline*}
\end{enumerate}
\end{mthem}
 If we do not assume the condition that $I$ has series Euler characteristic, the part 1 is given by the following.
 
 \begin{them}[Theorem \ref{any}]
Let $I$ be a finite category. Suppose the polynomial $|E-A_Iz|$ is factored to the following form:
$$|E-A_Iz|=d_{N-r}(z-\theta_1)^{e_1}\dots(z-\theta_n)^{e_n}$$ where $1\le r\le N-1$ each $e_i\ge 1$ and $\theta_i\not = \theta_j$ if $i\not = j$. Suppose 
$$\mathrm{sum}(\mathrm{adj}(E-A_I z)A_I)=q(z)|E-A_I z|+r(z)$$
where $$\deg(r(z))<\deg |E-A_Iz|$$
and $\frac{r(z)}{|E-A_Iz|}$ has a partial fraction decomposition to the following form 
$$\frac{r(z)}{ |E-A_I z|}=\frac{1}{d_{N-r} }\sum^n_{k=1}\sum^{e_k}_{j=1}\frac{A_{k,j}}{(z-\theta_k)^j}.$$ 
Then the zeta function of $I$ is 
\begin{multline*}
\zeta_I(z)=\prod^n_{k=1}\frac{1}{(1-\frac{1}{\theta_k}z)^{-\frac{A_{k,1}}{d_{N-r}}}}  \times \exp\bigg( Q(z)+  \\ \frac{1}{d_{N-r}}  \sum^n_{k=1} \sum^{e_k -1}_{j=1} \frac{z^j}{j(1-\frac{1}{\theta_k}z)^j}\bigg( \sum_{i=j}^{e_k-1}  \binom{i-1}{j-1} (-1)^{i-1} \bigg(\frac{1}{\theta_k}\bigg)^{i+j} A_{k,i+1}\bigg)\bigg)
\end{multline*}
where $Q(z)=\int q(z) dz$ is a polynomial whose constant term is $0$.
\end{them}

It is very important to study about behavior of singular points and zeros of a zeta function. By the following corollary, the problem is reduced to investigate properties of roots of $|E-A_I z|$. 
\begin{coro}
Let $I$ be a finite category. A complex number $z$ is a singular point or zero of $\zeta_I$ if and only if $z$ is a root of $|E-A_Iz|$.
\end{coro}

This paper is organized as follows.

In section \ref{preparation}, we prove some lemmas for a proof of our main theorem.

In section \ref{main-proof}, we prove our main theorem.

\section{Preparations for our main theorem}\label{preparation}

\subsection{Notation}

Throughout this paper, we will use the following notations.

\begin{enumerate}
\item We mean $I$ is a finite category which has $N$-objects.
\item The three polynomials $|E-A_Iz|$, $\mathrm{sum}(\mathrm{adj}(E-A_Iz))$ and $\mathrm{sum}(\mathrm{adj}(E-A_Iz)A_I)$ which will be often used later are expressed by the following form
$$|E-A_Iz|=d_0+d_1 z+\dots +d_N z^N,$$
$$\mathrm{sum}(\mathrm{adj}(E-A_Iz))=k_0+k_1 z+\dots+k_{N-1}z^{N-1}$$
and 
$$\mathrm{sum}(\mathrm{adj}(E-A_Iz)A_I)=m_0+m_1 z+\dots +m_{N-1} z^{N-1}.$$
\end{enumerate}
By Lemma 2.2 of \cite{NogB}, the degree of the third polynomial is less than or equal to $N-1$.
The coefficients $d_0, d_1$ and $d_N$ are $1, (-1)^N\mathrm{tr}(A_I)$ and $(-1)^N |A_I|$, respectively. Hence, the degree of $|E-A_I z|$ is larger or equal to 1 if $I$ is not an empty category since $\mathrm{tr}(A_I) \ge N$.

\subsection{Some lemmas}
In this subsection, we investigate the three polynomials above.

\begin{lemma}\label{d-lemma}
The degree of $|E-A_I z|$ is $N-r$ if and only if $|A_I-E z|$ can be divided by $z^r$, but can not be divided by $z^{r+1}$.
\begin{proof}
We have
$$|A_I-E z|=(-1)^N (d_0z^N+d_1 z^{N-1}+\dots+d_{N-1} z+d_N).$$
Indeed, if we write $$|A_I-Ez|=a_0+a_1z+\dots +a_Nz^N,$$
then we have
\begin{eqnarray*}
|E-A_I z|& = & (-1)^N z^N \bigg| A_I-E\frac{1}{z} \bigg| \\
&=&(-1)^N z^N \bigg(a_0+a_1\frac{1}{z}+\dots + a_N\frac{1}{z^N}\bigg)\\
&=&(-1)^N (a_0z^N+a_1 z^{N-1}+\dots +a_N) \\
&=&d_0+d_1z+\dots+d_N z^N.
\end{eqnarray*}
Hence, we have $a_0=(-1)^N d_N, a_1=(-1)^N d_{N-1},\dots, a_N=(-1)^N d_0$.

Suppose $\mathrm{deg}(|E-A_I z|)=N-r$. Then, $d_N=d_{N-1}=\dots=d_{N-r+1}=0$, but $d_{N-r}\not =0.$ Hence, we have
$$|A-E z|=(-1)^N d_0z^N+\dots +(-1)^N d_{N-r}z^{r}.$$
So $|A_I-E z|$ can be divided by $z^r$, but can not be divided by $z^{r+1}$.

Conversely, if the polynomial $|A_I-E z|$ can be divided by $z^r$, but can not be divided by $z^{r+1}$, then
$d_N=d_{N-1}=\dots=d_{N-r+1}=0$ and $d_{N-r}\not=0$. Hence, $\mathrm{deg}(|E-A_I z|)=N-r$.
\end{proof}
\end{lemma}

\begin{lemma}\label{k-lemma}
The degree of $\mathrm{sum}(\mathrm{adj}(E-A_I z))$ is $N-1-s$ if and only if $\mathrm{sum}(\mathrm{adj}(E-A_I z))$ can be divided by $z^s$, but can not be divided by $z^{s+1}$.
\begin{proof}
We have 
$$\mathrm{sum}(\mathrm{adj}(A_I -Ez))=(-1)^{N-1}(k_0z^{N-1}+k_1 z^{N-2}+\dots +k_{N-1}).$$
Indeed, if we write 
$$\mathrm{sum}(\mathrm{adj}(A_I -Ez ))=b_0+b_1z+\dots+b_{N-1}z^{N-1},$$
then we have
\begin{eqnarray*}
\mathrm{sum}(\mathrm{adj}(E-A_I z))&=&(-z)^{N-1} \mathrm{sum}\bigg(\mathrm{adj}\bigg(A_I -E \frac{1}{z}\bigg)\bigg) \\
&=&(-z)^{N-1}\bigg(b_0+b_1\frac{1}{z}+\dots+b_{N-1}\frac{1}{z^{N-1}}\bigg) \\
&=&(-1)^{N-1}b_0 z^{N-1}+\dots +(-1)^{N-1} b_{N-1}\\
&=&k_0+k_1z+\dots+k_{N-1}z^{N-1}.
\end{eqnarray*}
Hence, we have $b_0=(-1)^{N-1}k_{N-1}, b_{1}=(-1)^{N-1}k_{N-2}, \dots, b_{N-1}=(-1)^{N-1} k_0$.

Suppose $\mathrm{deg}( \mathrm{sum}(\mathrm{adj}(E-A_Iz))  )=N-1-s$. Then, $k_{N-1}=k_{N-2}=\dots=k_{N-s}=0$, but $k_{N-s-1}\not =0.$ Hence, we have
$$ \mathrm{sum}(\mathrm{adj}(E-A_Iz))  =(-1)^{N-1} k_0z^{N-1}+\dots +(-1)^{N-1} k_{N-1-s}z^{s}.$$
So $\mathrm{sum}(\mathrm{adj}(E-A_Iz))$ can be divided by $z^s$, but can not be divided by $z^{s+1}$.

Conversely, if the polynomial $\mathrm{sum}(\mathrm{adj}(E-A_Iz))$ can be divided by $z^s$, but can not be divided by $z^{s+1}$, then
$k_{N-1}=k_{N-2}=\dots=k_{N-s}=0$ and $k_{N-1-s}\not=0$. Hence, $\mathrm{deg}(\mathrm{sum}(\mathrm{adj}(E-A_Iz)) )=N-1-s$.
\end{proof} 
\end{lemma}

\begin{lemma}\label{classify}
Suppose the degree of $|E-A_I z|$ is $N-r$ and the degree of $\mathrm{sum}(\mathrm{adj}(E-A_I z))$ is $N-1-s$.
Then, $I$ has series Euler characteristic if and only if $s\ge r$. In this case, we have
$$\chi_{\Sigma}(I)=\begin{cases}0 &\text{if } s>r \\ -\frac{k_{N-1-s}}{d_{N-r}}&\text{if } s=r\end{cases}.$$
\begin{proof}

The finite category $I$ has series Euler characteristic if and only if the rational function
$$\frac{\mathrm{sum}(\mathrm{adj}(E-(A_I-E)t))}{|E-(A_I-E)t|}$$
can be substituted $-1$ to $t$ if and only if
the rational function
$$\frac{\mathrm{sum}(\mathrm{adj}(A_I-Ez))}{|A_I-Ez|}$$
can be substituted $0$ to $z$ (page 45 of \cite{Leib}).
Lemma \ref{d-lemma} and Lemma \ref{k-lemma} imply
\begin{eqnarray*}
\frac{\mathrm{sum}(\mathrm{adj}(A_I-Ez))}{|A_I-Ez|}&=&\frac{z^s h(z)}{z^r g(z)}
\end{eqnarray*}
for some polynomials $g(z)$ and $h(z)$ of $\mathbb{Z}[z]$ such that $g(z)$ and $h(z)$ can not divided by $z$. Hence, the rational function
$$\frac{\mathrm{sum}(\mathrm{adj}(A_I-Ez))}{|A_I-Ez|}$$
can be substituted $0$ to $z$ if and only if $s\ge r$. So the first claim is proved.

Suppose $I$ has series Euler characteristic. Then, we have $s\ge r$. If $s>r$, then it is clear $\chi_{\Sigma}(I)=0$. If $s=r$, then we have
\begin{eqnarray*}
\frac{\mathrm{sum}(\mathrm{adj}(A_I-Ez))}{|A_I-Ez|}&=&\frac{(-1)^{N-1}(k_0 z^{N-1} +k_1 z^{N-2}+\dots+k_{N-1-s} z^s)}{(-1)^N(d_0 z^N+d_1 z^{N-1}+\dots+d_{N-r}z^r)} \\
&=&-\frac{k_0z^{N-1-s}+\dots +k_{N-1-s}}{d_0z^{N-r}+\dots+d_{N-r}}.
\end{eqnarray*}
Hence, we obtain $\chi_{\Sigma}(I)=-\frac{k_{N-1-s}}{d_{N-r}}$.
\end{proof}
\end{lemma}

\begin{lemma}\label{deg-lemma}
If $I$ has series Euler characteristic, then we have $$\deg\big( \mathrm{sum}(\mathrm{adj}(E-A_Iz)A_I)\big)=\deg (|E-A_Iz|)-1 .$$
\begin{proof}
Lemma 2.2 of \cite{NogB} implies $$\mathrm{sum}(\mathrm{adj}(E-A_Iz)A_I)=\frac{1}{z}\bigg(\mathrm{sum}(\mathrm{adj}(E-A_Iz))-N |E-A_Iz |\bigg).$$
Note that the polynomial $$\mathrm{sum}(\mathrm{adj}(E-A_Iz))-N| E-A_Iz|$$ has no constant term since $k_0=N$ and $d_0=1$. Hence, we have $$\deg \big( \mathrm{sum}(\mathrm{adj}(E-A_Iz)A_I) \big) =\deg  \bigg(\mathrm{sum}(\mathrm{adj}(E-A_Iz))-N|E-A_Iz|\bigg)-1.$$
Since $I$ has series Euler characteristic, Lemma \ref{classify} implies $s\ge r$. Hence, we have the inequality $$\deg ( \mathrm{sum}(\mathrm{adj}(E-A_Iz)) )=N-1-s <N-r=\deg(|E-A_Iz|).$$
So we obtain
$$\deg\big( \mathrm{sum}(\mathrm{adj}(E-A_Iz)A_I)\big)=\deg (|E-A_Iz|)-1 .$$
\end{proof}
\end{lemma}

\begin{lemma}\label{mlemma}
If $I$ has series Euler characteristic and $\deg(|E-A_Iz|)=N-r$ and $$\deg ( \mathrm{sum}(\mathrm{adj}(E-A_Iz)) )=N-1-s,$$ then for the polynomial $$\mathrm{sum}(\mathrm{adj}(E-A_Iz)A_I)=m_0+m_1z+\dots+m_{N-1-r}z^{N-1-r},$$
we have $m_{N-1-r}=-N d_{N-r}$ and
$$m_{N-2-r}=\begin{cases}-Nd_{N-1-r} & \text{if} \; s>r \\ -Nd_{N-1-r}+k_{N-1-r} &\text{if}\; s=r .\end{cases}$$
\begin{proof}
Lemma 2.2 of \cite{NogB} implies \begin{eqnarray*}
\mathrm{sum}(\mathrm{adj}(E-A_I z)A_I)&=&\mathrm{sum}(\mathrm{adj}(E-A_I z)A_I)=m_0+m_1z\\
&&+\dots+m_{N-1-r}z^{N-1-r} \\
&=&\frac{1}{z}\bigg(\mathrm{sum}(\mathrm{adj}(E-A_Iz))-N |E-A_Iz |\bigg)\\
&=&\frac{1}{z}\bigg(k_0+k_1z+\dots+k_{N-1-s}z^{N-1-s} \\
&&-N\big(d_0+d_1z+\dots+d_{N-r}z^{N-r}\big) \bigg) \\
&=&(k_1-Nd_1)+(k_2-Nd_2)z+\dots \\
&&+(k_{n-1-s}-Nd_{N-1-s})z^{N-2-s}-Nd_{N-s}z^{N-1-s}\\
&&+\dots -Nd_{N-r}z^{N-1-r}.
\end{eqnarray*}
Since $I$ has series Euler characteristic, Lemma \ref{deg-lemma} implies $s\ge r$. Hence, $$N-1-s<N-r,$$ so that $m_{N-1-r}=-Nd_{N-r}$. 

If $s>r$, then $N-1-s<N-1-r$, so that $m_{N-2-r}=-Nd_{N-1-r}$.

If $s=r$, then $m_{N-2-r}=-Nd_{N-1-r}+k_{N-1-r}$.
\end{proof}
\end{lemma}


\section{A proof of main theorem}\label{main-proof}

\begin{them}\label{any}
Let $I$ be a finite category. Suppose the polynomial $|E-A_Iz|$ is factored to the following form:
$$|E-A_Iz|=d_{N-r}(z-\theta_1)^{e_1}\dots(z-\theta_n)^{e_n}$$ where $1\le r\le N-1$ each $e_i\ge 1$ and $\theta_i\not = \theta_j$ if $i\not = j$. Suppose 
$$\mathrm{sum}(\mathrm{adj}(E-A_I z)A_I)=q(z)|E-A_I z|+r(z)$$
where $$\deg(r(z))<\deg |E-A_Iz|$$
and $\frac{r(z)}{|E-A_Iz|}$ has a partial fraction decomposition to the following form 
$$\frac{r(z)}{ |E-A_I z|}=\frac{1}{d_{N-r} }\sum^n_{k=1}\sum^{e_k}_{j=1}\frac{A_{k,j}}{(z-\theta_k)^j}.$$ 
Then the zeta function of $I$ is 
\begin{multline*}
\zeta_I(z)=\prod^n_{k=1}\frac{1}{(1-\frac{1}{\theta_k}z)^{-\frac{A_{k,1}}{d_{N-r}}}}  \times \exp\bigg( Q(z)+  \\ \frac{1}{d_{N-r}}  \sum^n_{k=1} \sum^{e_k -1}_{j=1} \frac{z^j}{j(1-\frac{1}{\theta_k}z)^j}\bigg( \sum_{i=j}^{e_k-1}  \binom{i-1}{j-1} (-1)^{i-1} \bigg(\frac{1}{\theta_k}\bigg)^{i+j} A_{k,i+1}\bigg)\bigg)
\end{multline*}
where $Q(z)=\int q(z) dz$ is a polynomial whose constant term is $0$.
\begin{proof}
Since $$\deg(r(z))<\deg |E-A_Iz|,$$ we can have a partial fraction decomposition of the following form $$\frac{r(z)}{|E-A_I z|}=\frac{1}{d_{N-r}}\sum^n_{k=1}\sum^{e_k}_{j=1}\frac{A_{k,j}}{(z-\theta_k)^j}$$ for some complex numbers $A_{k,j}$. Hence, we have
\begin{eqnarray*}
\frac{\mathrm{sum}(\mathrm{adj}(E-A_I z)A_I}{|E-A_Iz|}&=&q(z)+\frac{r(z)}{|E-A_Iz|} \\
&=&q(z)+\frac{1}{d_{N-r} }\sum^n_{k=1}\sum^{e_k}_{j=1}\frac{A_{k,j}}{(z-\theta_k)^j}.
\end{eqnarray*}
Proposition 2.1 of \cite{NogB} implies
\begin{eqnarray*}
\zeta_I(z)&=&\exp\left(\int q(z) dz +\int \frac{1}{d_{N-r}}\sum^n_{k=1}\sum^{e_k}_{j=1}\frac{A_{k,j}}{(z-\theta_k)^j} dz \right)\\
&=&\exp\bigg( \int q(z) dz+\frac{1}{d_{N-r}}\int \sum^n_{k=1}\frac{A_{k,1}}{(z-\theta_k)} dz +\\
&& \frac{1}{d_{N-r}} \int \sum^n_{k=1}\sum^{e_k}_{j=2}\frac{A_{k,j}}{(z-\theta_k)^j} dz\bigg) \\
&=&\exp\bigg( Q(z)+\frac{1}{d_{N-r}} \sum^n_{k=1} A_{k,1}\log (z-\theta_k)  +\\
&& \frac{1}{d_{N-r}} \sum^n_{k=1}\sum^{e_k}_{j=2}-\frac{A_{k,j}}{(j-1)}\frac{1}{(z-\theta_k)^{j-1}} +C\bigg)\\
&=&\prod_{k=1}^n \frac{1}{(z-\theta_k)^{-\frac{A_{k,1}}{d_{N-r}}}}\times\\
&&\exp\bigg(Q(z)+\frac{1}{d_{N-r}} \sum^n_{k=1}\sum^{e_k-1}_{j=1}-\frac{A_{k,j+1}}{j}\frac{1}{(z-\theta_k)^{j}} \bigg)\exp{C}\\
&=&\prod_{k=1}^n \frac{1}{(-\theta_k)^{-\frac{A_{k,1}}{d_{N-r}}}(1-\frac{1}{\theta_k}z)^{-\frac{A_{k,1}}{d_{N-r}}}}\times\\
&&\exp\bigg(Q(z)+\frac{1}{d_{N-r}} \sum^n_{k=1}\sum^{e_k-1}_{j=1}-\frac{A_{k,j+1}}{j}\frac{1}{(z-\theta_k)^{j}} \bigg)C'\\
&=&\prod_{k=1}^n \frac{1}{(1-\frac{1}{\theta_k}z)^{-\frac{A_{k,1}}{d_{N-r}}}}\times\\
&&\exp\bigg(Q(z)+\frac{-1}{d_{N-r}} \sum^n_{k=1}\sum^{e_k-1}_{j=1}\frac{A_{k,j+1}}{j}\frac{1}{(z-\theta_k)^{j}} \bigg)C''
\end{eqnarray*}
where we did and will replace the constant term as $C,C'$ and $C''\dots$. Lemma 2.7 of \cite{NogB} implies
\begin{eqnarray*}
\zeta_I(z)&=&\prod_{k=1}^n \frac{1}{(1-\frac{1}{\theta_k}z)^{-\frac{A_{k,1}}{d_{N-r}}}}\times\\
&&\exp\bigg(Q(z)+\frac{-1}{d_{N-r}} \sum^n_{k=1}\sum^{e_k-1}_{j=1}\frac{A_{k,j+1}}{j} \sum_{i=1}^j \frac{\binom{j}{i}(-\frac{1}{\theta_k})^j (-z)^i}{(z-\theta_k)^i}\bigg)C'''
\end{eqnarray*}
Here, we use the boundary condition $\zeta_{I}(0)=1$. This condition is directly implied by the definition of the zeta function. Hence, we obtain $C'''=1$. By exchanging $\sum_i$ and $\sum_j$, we have
\begin{multline*}
\zeta_I(z)=\prod^n_{k=1}\frac{1}{(1-\frac{1}{\theta_k}z)^{-\frac{A_{k,1}}{d_{N-r}}}}  \\ \times \exp\left( Q(z)+\frac{1}{d_{N-r}}  \sum^n_{k=1} \sum^{e_k -1}_{j=1} \frac{z^j}{j(1-\frac{1}{\theta_k}z)^j}\bigg( \sum_{i=j}^{e_k-1}  \binom{i-1}{j-1} (-1)^{i+1} \bigg(\frac{1}{\theta_k}\bigg)^{i+j} A_{k,i+1} \bigg)\right)
\end{multline*}
Hence, we obtain the result.
\end{proof}
\end{them}

It is very important to study about behavior of singular points and zeros of a zeta function. By the following corollary, the problem is reduced to investigate properties of roots of $|E-A_I z|$.

\begin{coro}
Let $I$ be a finite category. A complex number $z$ is a singular point or a zero of $\zeta_I$ if and only if $z$ is a root of $|E-A_Iz|$.
\begin{proof}
Theorem \ref{any} directly implies all of the singular points and zeros are roots of $|E-A_Iz|$. Conversely, suppose $z$ is a root of $|E-A_Iz|$ but $z$ is not a singular point and a zero. Then, $z=\theta_{\ell}$ for some $\ell$. The index $-\frac{A_{\ell,1}}{d_{N-r}}$ must be 0. Namely, we have $A_{\ell,1}=0$. For $j=e_{\ell}-1$, $$\sum^{e_{\ell}-1}_{i=e_{\ell}-1}-A_{\ell,i+1}\binom{i-1}{j-1}(-\frac{1}{\theta_{\ell}})^{i-(e_{\ell}-1)}$$ must be 0 since $\zeta_I(z)$ is defined. Hence, we have $A_{\ell,e_{\ell}}=0$. As this, we can show each $A_{\ell,j}=0$ by the descent from $j=e_{\ell}-1$. Hence, we have 
\begin{eqnarray*}
\frac{r(z)}{|E-A_I z|}&=&\frac{1}{d_{N-r}}\sum^n_{k=1}\sum^{e_{\ell}}_{j=1}\frac{A_{k,j}}{(z-\theta_{\ell})^j}\\
&=&\frac{1}{d_{N-r}}\sum^n_{k=1,k\not=l}\sum^{e_{\ell}}_{j=1}\frac{A_{k,j}}{(z-\theta_{\ell})^j}
\end{eqnarray*}
Hence, we obtain
\begin{eqnarray*}
|E-A_I z|&=&d_{N-r} (z-\theta_1)^{e_1}\dots(z-\theta_n)^{e_n} \\
&=&d_{N-r} (z-\theta_1)^{e_1}\dots \\
&&(z-\theta_{\ell-1})^{e_{\ell-1}}(z-\theta_{\ell+1})^{e_{\ell+1}}\dots(z-\theta_n)^{e_n} 
\end{eqnarray*}
The polynomial $| E-A_I z|$ has two different degrees since each $e_k\ge 1$. This contradiction implies $z=\theta_{\ell}$ is a singular point or a zero of $\zeta_I$. 
\end{proof}
\end{coro}

\begin{them}\label{noguchi}
Suppose $I$ has series Euler characteristic and $$\deg(|E-A_Iz|)=N-r$$  and $$\deg(\mathrm{sum}(\mathrm{adj}(E-A_Iz)))=N-1-s$$ and the polynomial $|E-A_Iz|$ is factored to the following form
$$|E-A_Iz|=d_{N-r}(z-\theta_1)^{e_1}\dots(z-\theta_n)^{e_n}$$ where each $e_i\ge 1$ and $\theta_i\not = \theta_j$ if $i\not = j$. Then the rational function $$\frac{\mathrm{sum}(\mathrm{adj}(E-A_I z)A_I)}{ |E-A_I z|}$$ has a partial fraction decomposition to the following form 
$$\frac{\mathrm{sum}(\mathrm{adj}(E-A_I z)A_I)}{ |E-A_I z|}=\frac{1}{d_{N-r} }\sum^n_{k=1}\sum^{e_k}_{j=1}\frac{A_{k,j}}{(z-\theta_k)^j}$$ 
for some complex numbers $A_{k,j}$.
Moreover, 
\begin{enumerate}
\item \label{C1} the zeta function of $I$ is 
\begin{multline*}
\zeta_I(z)=\prod^n_{k=1}\frac{1}{(1-\frac{1}{\theta_k}z)^{-\frac{A_{k,1}}{d_{N-r}}}}  \\ \times \exp\left( \frac{1}{d_{N-r}}  \sum^n_{k=1} \sum^{e_k -1}_{j=1} \frac{z^j}{j(1-\frac{1}{\theta_k}z)^j}\bigg( \sum_{i=j}^{e_k-1}  \binom{i-1}{j-1} (-1)^{i-1} \bigg(\frac{1}{\theta_k}\bigg)^{i+j} A_{k,i+1}\bigg)\right)
\end{multline*}
\item the sum of all the indexes are the number of objects of $I$, that is, $$\displaystyle \sum^n_{k=1}-\frac{A_{k,1}}{d_{N-r}}=N$$ \label{C2}
\item \label{C3} Each $\frac{1}{\theta_k}$ is an eigen value of $A_I$. In particular, $\frac{1}{\theta_k}$ is an algebraic integer
\item \label{C4} 
\begin{multline}\label{main}
\displaystyle \sum^n_{k=1}\frac{-\frac{A_{k,1}}{d_{N-r} } }{\frac{1}{\theta_k}}\\
+ \frac{1}{d_{N-r}} \sum^n_{k=1} \sum ^{e_k-1}_{j=1} (-1)^j\frac{ \sum_{i=j}^{e_k-1}  \binom{i-1}{j-1} (-1)^{i-1} \bigg(\frac{1}{\theta_k}\bigg)^{i+j} A_{k,i+1}}{\left(\frac{1}{\theta_k}\right)^{j+1}}=\chi_{\Sigma}(I).
\end{multline}
\end{enumerate}
\end{them}
We give a simple interpretation of the part \ref{C4}. Put $\alpha_k=\frac{1}{\theta_k}$, $\beta_{k,0}=-\frac{A_{k,1}}{d_{N-r}}$ and 
$$\beta_{k,j}=\sum_{i=j}^{e_k-1}  \binom{i-1}{j-1} (-1)^{i-1} \bigg(\frac{1}{\theta_k}\bigg)^{i+j} A_{k,i+1}.$$
Then, the equation \eqref{main} is 
$$\sum_{k=1}^n \sum_{j=0}^{e_k -1}(-1)^j \frac{\beta_{k,j}}{\alpha_k^{j+1}}=\chi_{\Sigma}(I).$$
This theorem claims that this alternating sum is always a rational number and it is the series Euler characteristic $\chi_{\sum}(I)$ of $I$.

\begin{proof}[Proof of Theorem \ref{noguchi}] 
Lemma \ref{deg-lemma} implies
$$\deg(\mathrm{sum}(\mathrm{adj}(E-A_Iz)A_I))<\deg(|E-A_Iz|).$$Hence, we have a partial fraction decomposition
$$\frac{\mathrm{sum}(\mathrm{adj}(E-A_I z)A_I)}{ |E-A_I z|}=\frac{1}{d_{N-r} }\sum^n_{k=1}\sum^{e_k}_{j=1}\frac{A_{k,j}}{(z-\theta_k)^j}$$ for some complex numbers $A_{k,j}$.

The part \ref{C1} is directly implied by Theorem \ref{any} as $Q(z)=0$.

Next we show the part \ref{C2}. We observe the numerators of both sides $$\frac{\mathrm{sum}(\mathrm{adj}(E-A_I z)A_I)}{|E-A_I z|}=\frac{1}{d_{N-r}} \sum^n_{k=1}\sum^{e_k}_{j=1}\frac{A_{k,j}}{(z-\theta_k)^j}.$$
For the right hand side, when it is transformed to the left hand side by a reduction to common denominator, the coefficient of $z^{N-1-r}$ of the numerator is $\sum_{k=1}^n A_{k,1}$. Lemma \ref{mlemma} implies  $\sum_{k=1}^n A_{k,1}=m_{N-1-r}=d_{N-r}$. Thus, we obtain 
$$\sum^n_{k=1}-\frac{A_{k,1}}{d_{N-r}}=N.$$

 We show the part \ref{C3}. Since each $\theta_k$ is a root of the polynomial $|E-A_I z|$, we obtain
 \begin{eqnarray*}
 |E-A_I\theta_k |&=&0 \\
 (\theta_k)^N \bigg| E\frac{1}{\theta_k}-A_I \bigg |&=&0.
 \end{eqnarray*}
 Hence, $\frac{1}{\theta_k}$ is an eigen value of $A_I$. Note that $\theta_k\not =0$. Moreover, since $|E\lambda-A_I|$ is a monic polynomial with coefficients in $\mathbb{Z}$, $\frac{1}{\theta_k}$ is an algebraic integer.

 Finally, we show the part \ref{C4}. The equation \eqref{main} is
\begin{eqnarray*}
\eqref{main}&=&\displaystyle \sum^n_{k=1}\frac{-\frac{A_{k,1}}{d_{N-r}}}{\frac{1}{\theta_k}} \\
&&+ \frac{1}{d_{N-r}} \sum^n_{k=1} \sum ^{e_k-1}_{j=1} (-1)^j\frac{ \sum_{i=j}^{e_k-1}  \binom{i-1}{j-1} (-1)^{i-1} \bigg(\frac{1}{\theta_k}\bigg)^{i+j} A_{k,i+1}}{\left(\frac{1}{\theta_k}\right)^{j+1}} \\
&=&\sum^n_{k=1}\bigg(-\frac{\theta_k A_{k,1}}{d_{N-r}} \\
&&+\frac{1}{d_{N-r}} \sum^{e_k-1}_{j=1}\sum^{e_k-1}_{i=j} (-1)^{j}\bigg(-\frac{1}{\theta_k}\bigg)^{i-1}\binom{i-1}{j-1}A_{k,i+1} \bigg) \\
&=&\sum^n_{k=1}\bigg(-\frac{\theta_k A_{k,1}}{d_{N-r}} \\
&&+\frac{1}{d_{N-r}} \sum^{e_k-1}_{i=1} \bigg(-\frac{1}{\theta_k}\bigg)^{i-1}A_{k,i+1}\bigg(\sum^i_{j=1}(-1)^j\binom{i-1}{j-1}\bigg)  \bigg)\\
&=&\frac{1}{d_{N-r}} \bigg( \sum^n_{k=1}(-\theta_k A_{k,1}-A_{k,2})\bigg).
\end{eqnarray*}
 So it is enough to show 
\begin{eqnarray}
\frac{1}{d_{N-r}}\bigg( \sum^n_{k=1}-\theta_kA_{k,1}-A_{k,2}\bigg)=\chi_{\sum}(I).\label{main2}
\end{eqnarray}
By comparison of the numerators of both sides $$\frac{\mathrm{sum}(\mathrm{adj}(E-A_I z)A_I)}{|E-A_I z|}=\frac{1}{d_{N-r}} \sum^n_{k=1}\sum^{e_k}_{j=1}\frac{A_{k,j}}{(z-\theta_k)^j},$$
we have
$$m_{N-2-r}=\sum^n_{k=1} A_{k,2}-\sum^n_{k=1}A_{k,1}(\theta_1 e_1+\dots+\theta_k(e_k-1)+\dots \theta_n e_n).$$ Hence, the left hand side of \eqref{main2} is
\begin{eqnarray}
\frac{1}{d_{N-r}}\bigg( \sum^n_{k=1}-\theta_kA_{k,1}-A_{k,2}\bigg)&=&\frac{1}{d_{N-r}}\bigg( \sum^n_{k=1}-\theta_k A_{k,1}  -m_{N-2-r} \nonumber \\ 
&&-A_{k,1}(\theta_1 e_1+\theta_k(e_k-1)+\dots+\theta_n e_n)\bigg) \nonumber \\
&=&\frac{1}{d_{N-r}}\bigg(-\bigg(\sum^n_{k=1}A_{k,1}\bigg)\bigg(\sum^n_{k=1}\theta_k e_k\bigg) \nonumber \\
&&-m_{N-2-r}\bigg). \label{main3}
\end{eqnarray}
We have 
\begin{eqnarray*}
|E-A_Iz|&=&d_0+d_1z+\dots +d_{N-r}z^{N-r}\\
&=&d_{N-r}(z-\theta_1)^{e_1}\dots (z-\theta_n)^{e_n} \\
&=&d_{N-r}\bigg(z^{N-r}-\bigg(\sum^n_{k=1}\theta_k e_k\bigg)z^{N-1-r}+\dots \bigg).
\end{eqnarray*}
Hence, we obtain $-d_{N-r}(\sum^n_{k=1}\theta_k e_k)=d_{N-1-r}$, so that
$$-\sum^n_{k=1}\theta_k e_k=\frac{d_{N-1-r}}{d_{N-r}}.$$
We have already seen $\sum^n_{k=1} A_{k,1}=-N d_{N-r}$. Therefore, the equation \eqref{main3} is
\begin{multline}
\frac{1}{d_{N-r}}\bigg(-\bigg(\sum^n_{k=1}A_{k,1}\bigg)\bigg(\sum^n_{k=1}\theta_k e_k\bigg)-m_{N-2-r}\bigg)=\\ \frac{1}{d_{N-r}}\bigg( -Nd_{N-1-r}-m_{N-2-r}\bigg).\label{main4}
\end{multline}

Here we have to consider two cases $$\chi_{\Sigma}(I)=\begin{cases}0 &\text{if } s>r \\ -\frac{k_{N-1-s}}{d_{N-r}}&\text{if } s=r \end{cases}$$
(see Lemma \ref{classify}).

If $s>r$, Lemma \ref{mlemma} implies $m_{N-2-r}=-Nd_{N-1-r}$, so that the equation \eqref{main4} is
\begin{eqnarray*}
\frac{1}{d_{N-r}}\bigg( -Nd_{N-1-r}-m_{N-2-r}\bigg)&=&\frac{1}{d_{N-r}}\bigg( -Nd_{N-1-r}+Nd_{N-1-r}\bigg)\\
&=&0\\
&=&\chi_{\sum}(I).
\end{eqnarray*}
If $r=s$, Lemma \ref{classify} implies $m_{N-r-2}=k_{N-1-r}-Nd_{N-1-r}$. Hence, the equation \eqref{main4} is
\begin{eqnarray*}
\frac{1}{d_{N-r}}\bigg( -Nd_{N-1-r}-m_{N-2-r}\bigg)&=&\frac{1}{d_{N-r}}\bigg( -Nd_{N-1-r}-k_{N-1-r} +Nd_{N-1-r}\bigg)\\
&=& -\frac{k_{N-1-r}}{d_{N-r}}\\
&=&\chi_{\sum}(I).
\end{eqnarray*}
Hence, we obtain the results.
\end{proof}

\end{document}